\documentclass[11pt]{article}
\usepackage{pdfsync}
\usepackage{amsmath,amsthm,amsfonts,amssymb,mathrsfs,bm,wasysym}
\usepackage{epsfig}
\usepackage{color}
\usepackage{verbatim}
\usepackage{multicol}
\usepackage{comment}
\usepackage{graphicx}
\usepackage{centernot}
\usepackage{tikz}
\usepackage{color}
\usepackage{enumerate}
\usepackage{comment}
\usepackage{thmtools}
\usepackage{float}
\usepackage{caption}
\usepackage{subcaption}
\usepackage{textcomp}
\usepackage[normalem]{ulem}
\usepackage[utf8]{inputenc} 
\usepackage[T1]{fontenc}
\usepackage[color=green, textsize=tiny]{todonotes}
\usepackage{hyperref}
\usepackage{titlesec}
\usepackage{chngcntr}
\usepackage{cleveref}

\parskip \medskipamount
\titlespacing*{\subsubsection}
{0pt}{0.4\baselineskip}{0.2\baselineskip}

\declaretheoremstyle[
qed={//}
]{defstyle}

\newtheorem{theorem}{Theorem}[section]
\declaretheorem[style=defstyle,sibling=theorem]{definition}
\newtheorem{lemma}[theorem]{Lemma}
\newtheorem{proposition}[theorem]{Proposition}
\newtheorem{maintheorem}{Theorem}
\newtheorem{remark}[theorem]{Remark}

\counterwithout{equation}{section}

\newcommand{\pr}[1]{\mathbb{P}\!\left(#1\right)}
\newcommand{\prstart}[2]{\mathbb{P}_{#2}\!\left(#1\right)}
\newcommand{\prcond}[3]{\mathbb{P}_{#3}\!\left(#1\;\middle\vert\;#2\right)}

\newcommand{\E}[1]{\mathbb{E}\!\left[#1\right]}
\newcommand{\estart}[2]{\mathbb{E}_{#2}\!\left[#1\right]}

\newcommand{\vr}[1]{\mathrm{Var}\left(#1\right)}

\newcommand{\1}[1]{{\text{\Large $\mathfrak 1$}}_{#1}}

\newcommand{\Z}{\mathbb{Z}}
\newcommand{\Zpos}{\Z_{\ge1}}
\newcommand{\Znonneg}{\Z_{\ge0}}

\newcommand{\Gn}{G^{(n)}}
\newcommand{\Vn}{V^{(n)}}

\newcommand{\til}[1]{\widetilde{#1}}
\newcommand{\ha}[1]{\widehat{#1}}

\newcommand{\trel}{t_{\mathrm{rel}}}
\newcommand{\trelabs}{\trel^{\mathrm{abs}}}
\newcommand{\tmix}[1]{t_{\mathrm{mix}}\left(#1\right)}

\newcommand{\tmixl}[1]{t_{\mathrm{mix}}^{\mathrm{lazy}}\left(#1\right)}
\newcommand{\hit}[2]{\mathrm{hit}_{#1}\left(#2\right)}
\newcommand{\lhit}[2]{\mathrm{hit}^{\mathrm{lazy}}_{#1}\left(#2\right)}

\newcommand{\tcov}{t_{\mathrm{cov}}}

\newcommand{\taured}{\tau_{\mathrm{red}}}
\newcommand{\taublue}{\tau_{\mathrm{blue}}}
\newcommand{\taugeom}{\tau_{\mathrm{geom}}}

\newcommand{\Kroot}{K\mathrm{-root}}
\newcommand{\VKroot}{V_{\Kroot}}

\newcommand{\eps}{\varepsilon}

\newcommand{\tmixtext}{t_{\mathrm{mix}}}

\newcommand{\h}{\mathfrak{h}}
\newcommand{\V}{\mathfrak{V}}
\newcommand{\F}{\mathcal{F}}

\newcommand{\U}{\mathcal{U}}

\newcommand{\trunc}[2]{\mathrm{Tr}(#1,#2)}

\newcommand{\truncprime}[1]{\mathrm{Tr}'(#1)}

\newcommand{\Geompos}[1]{\mathrm{Geom}_{\ge1}\!\left(#1\right)}
\newcommand{\Geomnonneg}[1]{\mathrm{Geom}_{\ge0}\!\left(#1\right)}
\newcommand{\Bin}[2]{\mathrm{Bin}\!\left(#1,#2\right)}

\newcommand{\ent}[2]{\mathrm{H}_{#1}\left(#2\right)}
\newcommand{\entcond}[3]{\mathrm{H}_{#1}\left(#2\middle|#3\right)}

\newcommand{\dtv}[2]{d_{\mathrm{TV}}\left({#1},{#2}\right)}

\newcommand{\stle}{\le_{\mathrm{st}}}

\begin{document}

\title{Mixing time of the lazy and simple random walk on a randomly twisted hypercube}
\author{Zsuzsanna Baran
\thanks{Paris Dauphine University, Paris, France. {baran@ceremade.dauphine.fr}}
\and An{\dj}ela \v{S}arkovi\'c
\thanks{University of Cambridge, Cambridge, UK. {as2572@cam.ac.uk}}}
\date{}

\maketitle

\begin{abstract}
We consider a randomly twisted hypercube on $N$ vertices and study a simple or lazy random walk on this graph. We show that for both walks the mixing time is of order $\log N$, and the model does not exhibit cutoff. We also establish that both walks have cover time of order $N\log N$, and we show that the chromatic number of the graph diverges.

%The twisted hypercube model features edges corresponding to various random matchings, and inspired by~\cite{LRP} we use some of this randomness to create a framework where already established `quasi-tree techniques' can be used. Combined with some combinatorial ideas this leads to a short and neat proof.
\end{abstract}

\section{Introduction}

In this paper we consider the mixing properties of a (discrete-time) random walk on a twisted hypercube, which is a random generalisation of a Boolean hypercube, given by the following formal definition.% \magenta{I rephrased to avoid two as follows one after another; we can change..}

\begin{definition}\label{def:twisted_hypercube}
We recursively define the random graphs $\Gn$ for $n\ge0$ as follows. Let $G^{(0)}$ be a graph consisting of a single vertex, and for $n\ge1$ let $\Gn$ be obtained by considering two independent copies $G^{(n-1,1)}$ and $G^{(n-1,2)}$ of $G^{(n-1)}$ and adding the edges corresponding to a uniform random perfect matching between their vertices. Then we say that $\Gn$ is a \emph{twisted hypercube} on $2^n$ vertices.

We call the copies of $G^{(k)}$ appearing in the construction of $\Gn$ the \emph{type $k$ subgraphs} of $\Gn$, and we label them $G^{(k,1)}$, ..., $G^{(k,2^{n-k})}$ in such a way that for each $k\in\{1,2,...,n\}$ and $\ell\in\{1,2,...,2^{n-k}\}$ the type $(k-1)$ subgraphs of $G^{(k,\ell)}$ are labelled $G^{(k-1,2\ell-1)}$ and $G^{(k-1,2\ell)}$. We say that the edges running between $G^{(k-1,2\ell-1)}$ and $G^{(k-1,2\ell)}$ are \emph{type $k$ edges}.  (See~\Cref{fig:Gn} for illustration; disregard the colouring for now.)
\end{definition}

Since the 1990s, various modifications of a Boolean hypercube have been considered as models of networks that retain some of the structure of a hypercube, but have smaller diameter and better connectivity properties.

The above random model was first introduced by Dudek et al~\cite{randomly_twisted_hypercubes} in 2018, and they showed that with high probability between any pair of vertices there are at least $n$ internally disjoint paths of length at most $\frac{n}{\log_2n}+O\left(\frac{n}{(\log_2n)^2}\right)$. Among $n$-regular graphs with $2^n$ vertices, this number of paths is as large as possible, and the length is asymptotically optimal as the diameter of such a graph is at least $\frac{n}{\log_2n}$.

In 2023, Benjamini et al~\cite{twisted_hypercubes_structure_randomness} considered a more general random model, where the two copies of $G^{(n-1)}$ in the construction of $G^{(n)}$ do not need to be independent, and established high probability bounds on the diameter and vertex expansion, proved that its eigenvalues satisfy a semicircle law, and showed that it has no non-trivial automorphisms.

The paper~\cite{twisted_hypercubes_structure_randomness} also included a list of remarks and open questions (in Section 3). In particular, they asked about the mixing time of a lazy random walk on a twisted hypercube, noting that by simple estimates it is $O(n\log n)$ and $\Omega(n)$, and wondering whether it is $o(n\log n)$. They also asked about the chromatic number of a twisted hypercube, noting that with high probability it is at least 3.

%\magenta{They write that the mixing time is $\omega(n)$, but this is just a typo, right?}

In this work we study the mixing properties of a randomly twisted hypercube and we establish that both a lazy and a simple random walk have mixing time of order $n$. Using this we also show that the cover time is of order $n2^n$ for both walks.

Although it is not the focus of the paper, we also present a quick argument showing that the chromatic number is $\omega(1)$.

We recall the definitions of mixing time, cutoff, cover time and chromatic number.

\begin{definition}\label{def:tmix_cutoff}
Let $X$ be a Markov chain on a finite state space $V$, with a unique invariant distribution $\pi$. Then for $\eps\in(0,1)$ we define the $\eps$-mixing time of $X$ as
\[
\tmix{\eps}:=\quad\max_{x\in V}\:\inf\left\{t\ge0:\:\dtv{\prstart{X_t=\cdot}{x}}{\pi(\cdot)}\le\eps\right\}.
\]
where $d_{\mathrm{TV}}$ denotes the total variation distance between two distributions, defined as $\dtv{\mu}{\nu}=\frac12\sum_{x}|\mu(x)-\nu(x)|$.

%delete:

We say that a sequence $\left(X^{(n)}\right)$ of Markov chains exhibits cutoff at time $t_n$, with window $s_n$ if $\frac{s_n}{t_n}\to0$ as $n\to\infty$, and for any $\eps\in(0,1)$ there exists a constant $c(\eps)$ such that for all $n$ the $\eps$-mixing time of $X^{(n)}$ satisfies
\begin{equation}\label{eq:cutoff_def}
t_n-c(\eps)s_n\quad\le\quad\tmixtext^{(n)}\left(\eps\right)\quad\le\quad t_n+c(\eps)s_n.
\end{equation}
In case the Markov chains $X^{(n)}$ are random, we say that the sequence exhibits cutoff with high probability (whp) if~\eqref{eq:cutoff_def} holds with a probability $1-o(1)$ as $n\to\infty$.
%\magenta{We say that a random sequence $X^{(n)}$ of Markov chains with high probability does not exhibit cutoff if there exist $c>0$ and $\eps>0$ such that, with high probability $\frac{\tmixn{\eps}}{\tmixn{1-\eps}}\ge1+c$ holds.}
\end{definition}

\begin{definition}\label{def:tcov}
Let $X$ be a Markov chain on a finite state space $V$. Then the cover time of $X$ is defined as
\begin{align*}
\tcov:=\quad\max_{x\in V}\estart{\max_{y\in V}\tau_y}{x},
\end{align*}
where $\tau_y$ denotes the first time the chain visits $y$.
\end{definition}

\begin{definition}\label{def:chromatic_number}
Given a graph $G$, a colouring of its vertices with $k$ colours is a function $c:\:V(G)\to\mathcal{C}$ where $\mathcal{C}$ is a set of size $k$. We say that $c$ is a proper colouring if for any pair $\{x,y\}$ of neighbouring vertices, we have $c(x)\ne c(y)$. The chromatic number $\chi(G)$ of $G$ is the smallest positive integer $k$ such that $G$ has a proper colouring with $k$ colours.
\end{definition}

Our main results are the following.

\begin{maintheorem}\label{thm:main_mixing}
Let $\left(\Gn\right)$ be a sequence of twisted hypercubes as in~\Cref{def:twisted_hypercube} and let $X^{(n)}$ be a simple random walk on $\Gn$. Then with high probability the graphs $\Gn$ have the following property.

For any $\eps\in(0,1)$ there exist positive constants $c(\eps)$ and $C(\eps)$ such that for all sufficiently large $n$ the $\eps$-mixing time of $X^{(n)}$ satisfies
\[c(\eps)n\quad\le\quad\tmixtext^{(n)}\left(\eps\right)\quad\le\quad C(\eps)n.\]
Also, the sequence $\left(X^{(n)}\right)$ does not exhibit cutoff.

Moreover, the same statement also holds for a lazy random walk instead of a simple one.
\end{maintheorem}

\begin{maintheorem}\label{thm:main_cover}
There exist positive constants $c$ and $C$ with the following property. With high probability the twisted hypercube $\Gn$ is such that the cover time of a simple random walk on it satisfies
\[cn2^n\quad\le\quad\tcov^{(n)}\quad\le\quad Cn2^n.\]
Moreover, the same statement holds for a lazy random walk instead of a simple one.
\end{maintheorem}

%The proofs of both of these make use of the following auxiliary result.
%
%\begin{proposition}
%Let $\left(\Gn\right)$ be a sequence of twisted hypercubes and let $X^{(n)}$ be a simple or lazy random walk on $\Gn$. There exist positive constants $c$ and $C$ such that with high probability $\Gn$ has the following property. There exists a random time $\tau\le Cn$ such that for any vertices $x$ and $y$ we have
%\[\prcond{X_{\tau}=y}{\Gn}{x}\quad\ge\quad\frac{c}{2^n}\:.\]
%\end{proposition}
%
%Besides these, we also prove the following estimate on the chromatic number.

\begin{maintheorem}\label{thm:main_chromatic}
For any positive constant $c$, with high probability the chromatic number of $G^{(n)}$ satisfies $\chi\left(G^{(n)}\right)>c$.
\end{maintheorem}

\subsection{Notation}

We let $N=2^n$, and we write $\U$ to denote the uniform distribution on the vertices of $\Gn$.

We write SRW to mean a simple random walk.

For functions $a,b:\Znonneg\to[0,\infty)$ we write $a(n)\lesssim b(n)$ if there exists a constant $C>0$ such that for all sufficiently large $n$ we have $a(n)\le Cb(n)$, and write $a(n)\ll b(n)$ if for any constant $c>0$, for all sufficiently large $n$ (in terms of $c$) we have $a(n)\le cb(n)$. We define $\gtrsim$ and $\gg$ analogously. We write $a(n)\asymp b(n)$ if we have $a(n)\lesssim b(n)\lesssim a(n)$. Unless specified otherwise, these relations are considered as $n\to\infty$.

We often consider quantities like $0.9n$ or $C\log n$ and treat them as integers. In these cases one can take $\lfloor\cdot\rfloor$ or $\lceil\cdot\rceil$, but we omit this from the notation.

\subsection{Overview}

In this section we give an overview of the proof ideas of \Cref{thm:main_mixing}, which is the most involved one of our three results.
Instead of trying to estimate the mixing times of a walk directly, we bound the hit times of the walk and then compare the mixing times to these. The hit times are defined as follows.

\begin{definition}
Let $X$ be a Markov chain on a finite state space $V$ with a unique invariant distribution $\pi$. For each $\alpha,\theta\in(0,1)$ we define the corresponding \emph{hit time} of $X$ as
\[
\hit{\alpha}{\theta}:=\quad\inf\{t:\:\forall x,\:\forall A\subseteq V\text{ with }\pi(A)\ge\alpha\text{ we have }\prstart{\tau_A>t}{x}\le\theta\},
\]
where $\tau_A=\inf\left\{t:\:X_t\in A\right\}$.
\end{definition}

We will prove the following two results.

\begin{proposition}\label{hit_upper_bound}
With high probability there exist $\alpha,\theta\in(0,1)$ such that the corresponding hit time of a simple random walk on $\Gn$ satisfies $\hit{\alpha}{\theta}\lesssim n$.
\end{proposition}

\begin{proposition}\label{trelabs_order}
With high probability the absolute relaxation time of a simple random walk on $\Gn$ satisfies $\trelabs\asymp n$.
\end{proposition}

Once we have these, the proof of \Cref{thm:main_mixing} can be concluded quickly as follows.

\begin{proof}[Proof of \Cref{thm:main_mixing}]
From~\Cref{trelabs_order} we know that $\hit{\alpha}{\theta}\lesssim n$ for some $\alpha,\theta\in(0,1)$ and from~\cite[Proposition 1.5]{twisted_hypercubes_structure_randomness} we know that $\trel\asymp n$. Then \cite[Proposition 3.3 and Corollary 3.4]{characterisation_of_cutoff} allows us to bound the hit time with all parameters, and we get that $\hit{\beta}{\varphi}\lesssim n$ for all $\beta,\varphi\in(0,1)$.

Using~\cite[Lemma A.3]{weighted_random_matching} we can compare the hit times of a lazy random walk and a SRW, and get that $\lhit{\beta}{\varphi}\lesssim n$ for all $\beta,\varphi\in(0,1)$.

Then using~\cite[Proposition 1.8]{characterisation_of_cutoff}, the monotonicity of mixing time in the parameter, and that $\tmixl{\eps}\gtrsim\trel$ for all $\eps\in(0,1)$ (by~\cite[Theorem 12.4]{MTMC_book} for the lazy walk) we get that mixing time of the lazy walk satisfies $\tmixl{\eps}\asymp n\asymp\trel$ for all $\eps\in(0,1)$, hence by~\cite[Theorem 3]{characterisation_of_cutoff} the lazy walk does not exhibit cutoff.

Since the lazy walk does not exhibit cutoff, neither does the simple random walk. (See~\cite[Lemma 6.9]{weighted_random_matching}.)

Using the above bound on the hit times, \cite[Remark 1.9]{characterisation_of_cutoff}, and that $\tmix{\eps}\gtrsim\trelabs$ for all $\eps\in(0,1)$ (see~\cite[Theorem 12.4]{MTMC_book}), we also get that $\tmix{\eps}\asymp n$ for all $\eps\in(0,1)$.
\end{proof}

To prove~\Cref{hit_upper_bound}, we define a random time $\tau$ and show that with positive probability it satisfies $\tau\lesssim n$, and $X_{\tau}$ is comparable to the uniform distribution.

Inspired by~\cite{LRP}, we colour the `short', `medium length' and `long' edges of $\Gn$ green, red, and blue respectively. We consider a random time $\tau$ so that the first $\tau$ steps of a random walk $X$ on $\Gn$ roughly speaking correspond to taking a geometric number of steps as a random walk on the green and red edges, then walking on the green and red edges until $K+L$ further red edges are crossed (where $K$ and $L$ are well-chosen parameters, analogous to the parameters with these labels used in \cite{LRP} and other works with similar approach), then crossing a blue edge, then walking on the green and red edges again until $L+K$ more red edges are crossed.

If after the initial geometric number of steps the walk is at a given vertex $q$, we approximate the probability of $X_{\tau}=y$ as
\begin{align}\label{eq:sum_overview}
\sum_{u,v,w,z}\prcond{Y_{\taured^K}=u}{\Gn}{q}&\prcond{Y_{\taured^L-1}=w}{\Gn}{u}\\
\nonumber
&\cdot\prcond{Y_{\taured^K}=v}{\Gn}{y}\prcond{Y_{\taured^L-1}=z}{\Gn}{v}\1{\eta(w)=z},
\end{align}
where $Y$ is a simple random walk on the green and red edges of $\Gn$, $\taured^i$ denotes the $i$th time that a red edge is crossed, and $\eta(w)$ is a uniformly chosen blue neighbour of $w$.

To approximate the neighbourhoods of $u$ and $v$ formed by green and red edges of $\Gn$, we consider random `quasi-trees' $T_u$ and $T_v$ that consist of `balls'  that are distributed like balls of a given radius $R$ in the green distance in a random copy of $\Gn$, and `long-range edges' between them that correspond to the red edges of $\Gn$ and form a tree-like structure.

Using the iid structure of a quasi-tree $T$, and the fact that due to the high degrees of vertices, a random walk $\til{Y}$ on it is unlikely to return to a ball it already left, we can get a good understanding of the behaviour $\til{Y}$. We then prove that for a sufficiently large proportion of pairs $(u,v)$, their neighbourhood in the green and red edges of $\Gn$, and a simple random walk on it is well approximated by a random quasi-tree and simple random walk on it. This allows us to get a good control over the probabilities $\prcond{Y_{\taured^L-1}=w}{\Gn}{u}$ and $\prcond{Y_{\taured^L-1}=z}{\Gn}{v}$. After this we use the randomness of the blue edges and a concentration result similar to~\cite[Lemma 5.1]{cutoff_NBRW_on_sparse_random_graphs} to show that the sum~\eqref{eq:sum_overview} is $\gtrsim\frac{1}{N}$.

To prove \Cref{trelabs_order}, we use a result that roughly speaking states that if a graph is `far from being bipartite', then its relaxation time and absolute relaxation time are of the same order. We prove that whp $\Gn$ satisfies the above property, and use that $\trel\asymp n$ from~\cite[Proposition 1.5]{twisted_hypercubes_structure_randomness}.

\subsection{Organisation}

In Sections \ref{sec:tau} to \ref{sec:hit} we work towards the proof of \Cref{hit_upper_bound}. In \Cref{sec:tau} we introduce the colouring of the edges of $\Gn$ and the exact definition of the random time $\tau$. In \Cref{sec:quasi-tree} we give the definition of a quasi-tree and explain how to couple $\Gn$ to it. In \Cref{sec:success} we estimate the success probability of this coupling. In \Cref{sec:hit} we conclude the proof of \Cref{hit_upper_bound}.

In \Cref{sec:trelabs} we prove \Cref{trelabs_order}, hence concluding the proof of \Cref{thm:main_mixing}. In \Cref{sec:cover} we prove \Cref{thm:main_cover}, and finally in \Cref{sec:chromatic} we prove \Cref{thm:main_chromatic}.

%In \Cref{sec:general_case} we briefly comment on how our results could be generalised to other types of twisted hypercubes, in particular for the so-called duplicube.

\section{The random time $\tau$}\label{sec:tau}

We colour the edges of $\Gn$ with three different colours as follows. We colour the type $n$ edges of $\Gn$ blue, the type $k$ edges for $k\in\{0.9n,...,n-1\}$ red, and all remaining edges green. (See~\Cref{fig:Gn} for an illustration.)

\begin{figure}[h]
\centering
\includegraphics[width=120mm]{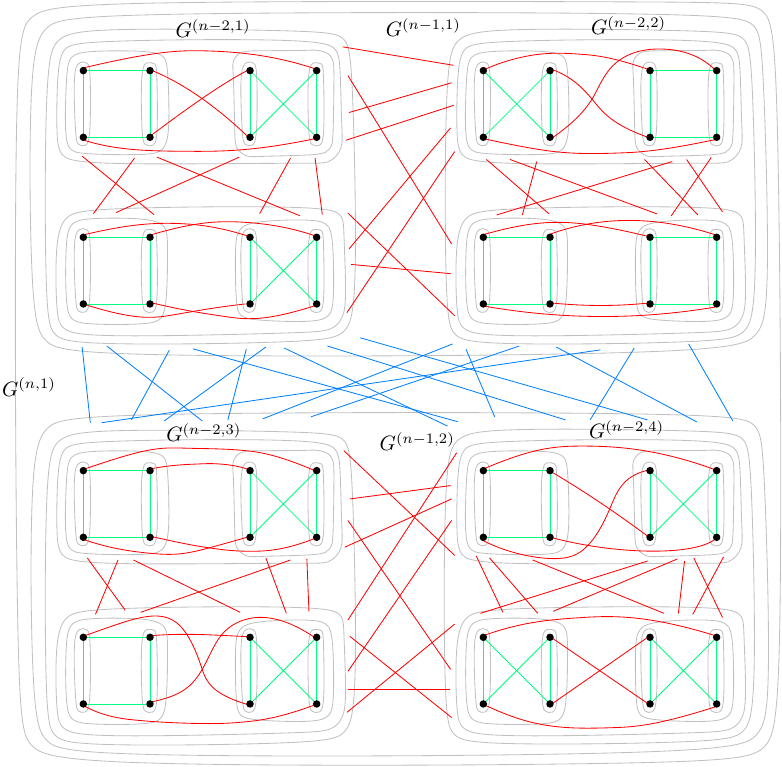}
\caption{An illustration of $\Gn$, the naming of its subgraphs, and the colouring of its edges.}
\label{fig:Gn}
\end{figure}

Given $\Gn$, we define a random time $\tau$ and an auxiliary walk $\til{X}$ as follows.

\begin{definition}
Let
\begin{equation}\label{eq:LK}
L:=\quad C_L\frac{n}{\log n}+C_{L,2}\sqrt{\frac{n}{\log n}},\qquad\quad K:=\quad C_K,
\end{equation}
where $C_L$ is an appropriately chosen constant (see choice after~\Cref{entropy_concentration}) and $C_{L,2}$ and $C_K$ are sufficiently large constants, to be specified later.

Let $\taured^{(L+K,j)}=g^{(1,j)}+...+g^{(L+K,j)}+(L+K-1)$ for $j=1,2$ where $g^{(i,j)}$ are independent $\Geomnonneg{\frac{0.1n}{n-1}}$ random variables and let $\taugeom\sim\Geomnonneg{\frac1n}$, independently of these. Let
\[\tau:=\quad\taugeom+1+\taured^{(L+K,1)}+1+\taured^{(L+K,2)}.\]

Let us fix a vertex $x$ of $\Gn$, and let $\til{X}$ be a walk on $\Gn$ defined up to time $\tau$, taking steps as follows. The first $\taugeom$ steps are a SRW from $x$ on the green and red edges of $\Gn$ (i.e.\ on the edges of the type $(n-1)$ subgraph containing $x$). Then for $i=1,2,...,L+K$ it crosses a uniformly chosen red edge emanating from its current position, then it takes $g^{(i,1)}$ steps as a SRW on the green edges. Then it crosses the blue edge emanating from its current position. Then for $i=L+K,...,2$ it takes $g^{(i,2)}$ steps as a SRW on the green edges and crosses a uniformly chosen red edge from its location. Finally, it takes $g^{(1,2)}$ steps as a SRW on the green edges. See~\Cref{fig:tau} for an illustration.
\end{definition}

\begin{figure}[h]
\centering
\includegraphics[width=150mm]{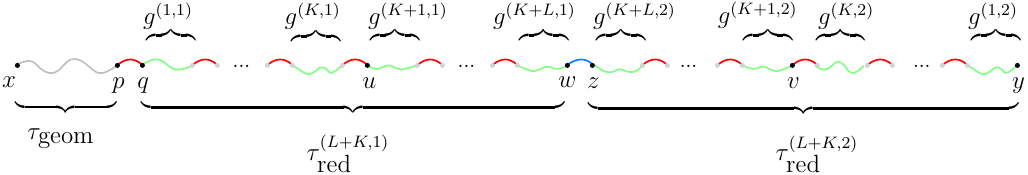}
\caption{An illustration of the walk $\til{X}$ up to time $\tau$.}
\label{fig:tau}
\end{figure}

The proof of \Cref{hit_upper_bound} relies on the following three results.

The distribution of $\til{X}_{\tau}$ is comparable to the uniform distribution $\U$.

\begin{proposition}\label{transition_probs}
There exists a constant $\theta$ such that with high probability $\Gn$ has the following property.
For any vertices $x$ and $y$ of $\Gn$ that are in different type $(n-1)$ subgraphs, we have
\begin{align*}
\prcond{\til{X}_{\tau}=y}{\Gn}{x}\quad\ge\quad\frac{\theta}{N}.
\end{align*}
\end{proposition}

The time $\tau$ is likely of order at most $n$.

\begin{lemma}\label{tau_upper_bound}
For any $\theta\in(0,1)$, for a sufficiently large constant $C$, for any realisation of the graph $\Gn$ and any vertex $x$, we have $\prcond{\tau>Cn}{\Gn}{x}\le\theta$.
\end{lemma}

The walk $\til{X}$ can be coupled closely with a simple random walk on $\Gn$.

\begin{lemma}\label{couple_X_Xtil}
For any choice of $C_L$, $C_{L,2}$ and $C_K$, any realisation of the graph $\Gn$ and any vertex~$x$, there exists a coupling between the pair $\left(\tau,\til{X}\right)$ defined as above, and a simple random walk $X$ on $\Gn$ starting from $x$, such that with probability $1-o(1)$ the walks $X$ and $\til{X}$ agree up to time $\tau$. Here the $o(1)$ can depend on the parameters $C_L$, $C_{L,2}$ and $C_K$, but it does not depend on the realisation of $\Gn$ and the choice of the vertex $x$.
\end{lemma}

The proof of \Cref{tau_upper_bound} is very simple and omitted. The proof of \Cref{couple_X_Xtil} is also quick, and it is presented below.

\begin{proof}[Proof of \Cref{couple_X_Xtil}] %\magenta{[Double-check this proof.]}
Note that a simple random walk $X$ on $\Gn$ at each step crosses a blue edge with probability $\frac1n$, independently for different steps, and conditional on which steps are blue, the remaining steps are green with probability $\frac{0.9n-1}{n-1}$ and red otherwise, independently of each other. Because of this we can sample a simple random walk $X$ on $\Gn$ up to the second time it crosses a blue edge as follows.

Let $\taublue^{(1)},\taublue^{(2)}\sim\Geompos{\frac1n}$ and let $g^{(i,j)}\sim\Geomnonneg{\frac{0.1n}{n-1}}$, all independently of each other.

%Let $B^{(j)}=\taublue^{(j)}$, let $G^{(j)}_1=g^{(1,j)}\wedge(B^{(j)}-1)$, $R^{(j)}_1=1\wedge(B^{(j)}-1-G^{(j)}_1)$, $G^{(j)}_2=g^{(2,j)}\wedge(B^{(j)}-1-G^{(j)}_1-R^{(j)}_1)$, ... $R^{(j)}_{L+K}=1\wedge(B^{(j)}-1-G^{(j)}_1-R^{(j)}_1...-G^{(j)}_{L+K})$, $S^{(j)}=B^{(j)}-1-G^{(j)}_1-...-R^{(j)}_{L+K}$. 
%
%Let $X$ take $S^{(1)}$ steps as a SRW on the green and red edges, then for $i=L+K,L+K-1,...,1$ let it take $R^{(1)}_i$ steps as a SRW on the red edges, then $G^{(1)}_i$ steps as a SRW on the green edges. Then let it cross a blue edge. Then for $i=1,2,...,L+K$ let it take $G^{(2)}_i$ steps as a SRW on the green edges, then $R^{(2)}_i$ steps as a SRW on the red edges. Finally, let it take $S^{(2)}$ steps as a SRW on the green and red edges, and then let it cross a blue edge.

%\magenta{I think we might need to do this instead for the above paragraphs: 

Let $B^{(j)}=\taublue^{(j)}$, let $G^{(j)}_{K+L}=g^{(K+L,j)}\wedge(B^{(j)}-1)$, $R^{(j)}_{K+L}=1\wedge(B^{(j)}-1-G^{(j)}_{K+L})$, $G^{(j)}_{K+L-1}=g^{(K+L-1,j)}\wedge(B^{(j)}-1-G^{(j)}_{K+L}-R^{(j)}_{K+L})$, $\ldots$, $R^{(j)}_{1}=1\wedge(B^{(j)}-1-G^{(j)}_{K+L}-R^{(j)}_{K+L}...-G^{(j)}_{1})$, $S^{(j)}=B^{(j)}-1-G^{(j)}_1-...-R^{(j)}_{L+K}$. We note that none of the quantities defined above can be negative. Intuitively, we define $G^{(j)}_i$ and $R^{(j)}_i$ starting from $i=L+K$ going down to $1$, alternating between the two, to be $g^{(i,j)}$, and $1$, respectively, as long as the sum of all of the so far defined parameters is at most $B^{(j)}-1$, while if at some point the sum of so far defined parameters becomes at least $B^{(j)}-1$ we set the parameter we are currently defining so that the sum is exactly $B^{(j)}-1$ and define the remaining parameters to be $0$.

Let $X$ take $S^{(1)}$ steps as a SRW on the green and red edges, then, for $i=1, 2,\ldots, L+K$ let it take $R^{(1)}_i$ steps as a SRW on the red edges, then $G^{(1)}_i$ steps as a SRW on the green edges. Then let it cross a blue edge. Then for $i=L+K,\ldots,1$ let it take $G^{(2)}_i$ steps as a SRW on the green edges, then $R^{(2)}_i$ steps as a SRW on the red edges. Finally, let it take $S^{(2)}$ steps as a SRW on the green and red edges, and then let it cross a blue edge.
%}

Having this representation of $X$, let us couple it to $\til{X}$ as follows. Let us sample $g^{(i,j)}$, and conditional on these, let us couple $\taublue^{(1)}\sim\Geomnonneg{\frac1n}$ and $\taugeom+\taured^{(L+K,1)}+1\sim\Geompos{\frac1n}+g^{(1,1)}+...+g^{(L+K,1)}+L+K+1$ using their optimal coupling, so that they agree on the event $\left\{\taublue^{(1)}\ge g^{(1,1)}+...+g^{(L+K,1)}+L+K+1\right\}$.

Then on the event $\left\{\taublue^{(j)}\ge g^{(1,j)}+...+g^{(L+K,j)}+L+K+1\text{ for }j=1,2\right\}$, the walks $X$ and $\til{X}$ agree up to time $\tau$.

This event has probability $1-o(1)$, which we can see as follows. We have $1-\dtv{\taublue^{(1)}}{\taugeom+m+1}=\sum_{k\ge0}\pr{\taublue^{(1)}=k}\wedge\pr{\taugeom+m+1=k}=\left(1-\frac1n\right)^{m+1}$, so conditional on $\taured^{(L+K,1)}=m$, the optimal coupling fails with probability $1-\left(1-\frac1n\right)^{m+1}$. Taking expectation over $\taured^{(L+K,1)}$ and using the independence of the variables $g^{(i,j)}$ we get that the overall failure probability is\\ $\le1-\E{\left(1-\frac1n\right)^{g(1,1)+1}}^{L+K}\left(1-\frac1n\right)^2\le1-e^{-c/\log n}$ where $c>0$ is a constant. This is $o(1)$ as required. This finishes the proof.
\end{proof}

In the rest of this section and the following sections, we will be working towards the proof of \Cref{transition_probs}. The proof of \Cref{hit_upper_bound} given \Cref{transition_probs} and Lemmas~\ref{tau_upper_bound} and~\ref{couple_X_Xtil} is presented in \Cref{sec:hit}.

For any two vertices $x$ and $y$ that are not in the same type $(n-1)$ subgraph of $\Gn$, we can express the transition probability $\prcond{\til{X}_{\tau}=y}{\Gn}{x}$ as follows.

Let $Y$ be a SRW on the green and red edges of $\Gn$, independently of $\taugeom$, and let $\taured^{i}$ denote the $i$th time that $Y$ crosses a red edge. For a vertex $w$ of $\Gn$, let $\eta(w)$ denote its blue neighbour. For vertices $p,q$ of $\Gn$ let $\{p\sim_r q\}$ denote the event that $p$ and $q$ are connected via a red edge. Then we have
\begin{align}\label{eq:sum}
\prcond{\til{X}_{\tau}=y}{\Gn}{x}&\\
\nonumber
=\:\sum_{p,q,u,w,z,v}&\prcond{Y_{\taugeom}=p}{\Gn}{x}\frac{\1{p\sim_{r}q}}{0.1n}
\prcond{Y_{\taured^K}=u}{\Gn}{q}\prcond{Y_{\taured^K}=v}{\Gn}{y}\\
\nonumber
&\hspace{2cm}\cdot\prcond{Y_{\taured^L-1}=w}{\Gn}{u}\prcond{Y_{\taured^L-1}=z}{\Gn}{v}\1{\eta(w)=z},
\end{align}
where the sums are taken over all vertices of $\Gn$.

In what follows, we will show that for most $u$ and $v$ in different type $(n-1)$ subgraphs, their neighbourhoods according to the green and red edges, and the walks on these neighbourhoods can be approximated well by some random tree-like structures (quasi-trees) and walks on these. We use this to get a good control over $\prcond{Y_{\taured^L-1}=w}{\Gn}{u}$ and $\prcond{Y_{\taured^L-1}=z}{\Gn}{v}$. Then we use the randomness of the blue edges to show that $\sum_{w,z}\prcond{Y_{\taured^L-1}=w}{\Gn}{v}\prcond{Y_{\taured^L-1}=z}{\Gn}{y}\1{\eta(w)=z}$ is concentrated.

To extend the result to all $x$ and $y$ in different type $(n-1)$ subgraphs, we show that for any starting point, with high probability $Y_{\taured^K}$ is such that the above approximation can be applied for its neighbourhood.

\section{The random quasi-tree and its coupling to $\Gn$} \label{sec:quasi-tree}

In this section we define random quasi-trees and explain how to we use these to approximate the green and red neighbourhoods of two vertices in $\Gn$.

\subsection{Definition of a quasi-tree}

Analogously to~\cite{LRP} we consider a random quasi-tree with the green edges forming the balls, and the red edges serving as long-range edges.

Let
\begin{equation}\label{eq:R}
R:=\quad C_R\log n,
\end{equation}
where $C_R$ is a sufficiently large constant, to be specified later.

\begin{definition}\label{def:quasi-tree}
A \emph{random quasi-tree} $T$ is a random graph together with a map $\iota$ from its vertices to the vertex set $\Vn$ of $\Gn$, obtained as follows. Let us start from a root vertex $\rho$, with a given $\iota(\rho)$. Then consider a ball of radius $R$ around $\iota(\rho)$ in a random copy of the type $(0.9n-1)$ subgraph containing $\iota(\rho)$, and attach it to $\rho$. We call this an \emph{$R$-ball}. For each vertex $v$ in the ball let $\iota(v)$ be the corresponding vertex in $\Gn$. Then for each vertex $v$ in the ball, except for $\rho$ we add $0.1n$ edges leading to new vertices $w_{v,0.9n},...,w_{v,n-1}$, while for $v=\rho$ we add $0.1n-1$ new edges leading to new vertices $w_{v,0.9n+1},...,w_{v,n-1}$. We call these \emph{long-range edges}, and in particular call $(v,w_{v,k})$ a type $k$ long-range edge. For each $w_{v,k}$, we sample $\iota(w_{v,k})$ by choosing uniformly from the $2^k$ vertices of $\Gn$ with the property that they are in the same type $(k+1)$ subgraph as $\iota(v)$ but in different type $k$ subgraphs. Analogously to the above, we consider a ball of radius $R$ in an independent random copy of the type $(0.9n-1)$ subgraph containing $\iota(w)$, attach it to $w$ and let $\iota$ map each of its vertices to the corresponding vertex of $\Gn$. We continue analogously, but if $v$ is the new endpoint of a newly revealed long-range edge of type $k$, then the further long-range edges we add from $v$ are of types $\{0.9n,...,n-1\}\setminus\{k\}$. This way we obtain an infinite random graph consisting of $R$-balls and long-range edges between them, where apart from $\rho$, each vertex is the endpoint of exactly $0.1n$ long-range edges which are of types $0.9n,...,n-1$.
\end{definition}

Note that the edges of the $R$-balls correspond to green edges, and the long-range edges correspond to red edges, but the $R$-balls and the long-range edges are always sampled from a new copy of $\Gn$.

For a vertex $v$ in the quasi-tree, we define its level as the minimal number of long-range edges a path from $\rho$ to $v$ needs to cross. For a long-range edge $e$, we define its level as the level of its endpoint that is further from $\rho$, and we denote it by $\ell(e)$. 

For a vertex $v$ of the quasi-tree $T$, we let $T(v)$ denote the subgraph of $T$ induced by the $R$-ball of $v$ and the $R$-balls descending from this.

Let $Z$ be a walk on the quasi-tree that takes steps as a SRW, except that it is not allowed to cross a long-range in the direction towards to root.\footnote{More precisely, if $Z$ is at $\rho$ or at a vertex that is not the centre of an $R$-ball, it takes a step to a uniform neighbour. Otherwise it disregards the one neighbour that has a smaller level than the current vertex, chooses uniformly from the others, and takes a step there.} Let $\xi_k$ be the $k$th long-range edge crossed by $Z$.

If $Z$ starts from $\rho$, then the entropy of $\xi_k$ is concentrated in the following sense.

\begin{lemma}\label{entropy_concentration}
For any $\theta>0$, there exists a constant $C>0$ with the following property. Let $T$ be a random quasi-tree and let $\xi^{(1)}$ and $\xi^{(2)}$ be two independent copies of $\xi$ on $T$. Then for any $k$ we have
\[
\pr{\left|-\log\prcond{\xi^{(1)}_k=\xi^{(2)}_k}{\xi^{(1)},T}{\rho}-k\h\right|>C\sqrt{k\V}}\quad\le\quad\theta,
\]
where $\h=c_1\log n$, $\V=c_2(\log n)^2$ and $c_1,c_2\asymp1$ as $n\to\infty$.
\end{lemma}

\begin{proof} Let us use the notation $\ent{b}{\cdot}$ and $\entcond{b}{\cdot}{\cdot}$ from~\cite[Definitions 2.23 and 2.24]{weighted_random_matching}, and their properties from Section B of the same article.

Then we can note that for any realisation $t$ of $T$, we have $\entcond{1}{\xi_1}{T=t}\asymp\log n$ and $\entcond{2}{\xi_1}{T=t}\asymp(\log n)^2$, and the values of these only depend on the $R$-ball of $\rho$. \footnote{We drop the conditioning on $T=t$ from the notation. Let $\taured$ be the first time that $Z$ crosses a long-range edge. Then $\taured\stle\Geompos{\frac{0.1n}{n-1}}$, hence $\ent{1}{\taured}\asymp\ent{2}{\taured}\asymp1$. Also, $Z_k$ can take $\le n^k$ different values, hence $\entcond{1}{Z_k}{\taured=k}\le k\log n$ and $\entcond{2}{Z_k}{\taured=k}\le (k\log n)^2$. Conditional on the value of $\taured$ and $Z_{\taured-1}$, the value of $\xi_1$ is uniform over $0.1n$ or $0.1n-1$ different edges, hence we have $\entcond{1}{\xi_1}{Z_{\taured}=v,\taured=k}\asymp\log n$ and $\entcond{2}{\xi_1}{Z_{\taured}=v,\taured=k}\asymp(\log n)^2$ for any possible $v$ and $k$. Using these and that $\entcond{b}{\xi_1}{Z_{\taured},\taured}-O(1)\lesssim\ent{b}{\xi_1}\lesssim\entcond{b}{\xi_1}{Z_{\taured},\taured}+\entcond{b}{Z_{\taured}}{\taured}+\ent{b}{\taured}$ we get the result.}

Since the $R$-balls are iid and the walk $Z$ does not backtrack long-range edges, we get that\\ $-\log\prcond{\xi^{(1)}_k=\xi^{(2)}_k}{\xi^{(1)},T}{\rho}$ decomposes as a sum of iid copies of $-\log\prcond{\xi^{(1)}_1=\xi^{(2)}_1}{\xi^{(1)},T}{\rho}$. We can also note that $\vr{-\log\prcond{\xi^{(1)}_1=\xi^{(2)}_1}{\xi^{(1)},T}{\rho}}\le\entcond{2}{\xi_1}{T}$, hence for any $\theta$, for sufficiently large $C$ we have
\[
\pr{\left|-\log\prcond{\xi^{(1)}_k=\xi^{(2)}_k}{\xi^{(1)},T}{\rho}-k\h\right|>C\sqrt{k\V}}\quad\le\quad\theta
\]
where $\h=\entcond{1}{\xi_1}{T}=c_1\log n$, $\V=\entcond{2}{\xi_1}{T}=c_2(\log n)^2$ with $c_1,c_2\asymp1$ as $n\to\infty$.
\end{proof}

Let us choose the constant $C_L$ in~\eqref{eq:LK} as $C_L=\frac{\log 2}{2c_1}$.

\subsection{Coupling to $\Gn$}\label{sec:trunc_Kroot_coupling}

In this section we explain how we couple the (green and red) neighbourhoods of two vertices of $\Gn$ and the walks on them with random quasi-trees and walks on them.

Similarly to previous works, we are not able to closely couple the neighbourhood of a vertex $v$ in $\Gn$ to the entire first $L$ levels of a quasi-tree $T$, since the chance that some vertices appear multiple times in the $\iota$ of the first $L$ levels of $T$ is too high. Instead, we focus on coupling the parts that are likely to be visited by a walk. If a long-range edge of $T$ is too unlikely to be crossed by $Z$, then we disregard the part of $T$ descending from it (we `truncate' that edge). Below we define `truncation criteria' that specify more precisely when to do this. Also, if we already have a disagreement between some parts of $T$ and $\Gn$, then we do not continue trying to couple the parts of $T$ descending from it.

We also restrict our attention to vertices $v$ of $\Gn$ that are `$K$-roots', i.e.\ where we already know that their neighbourhood is tree-like in the first $K<L$ levels, and we only couple with $T$ from then onwards.

Once we coupled the neighbourhood of $v$ with a tree $T$, we also couple the walks $Y$ and $Z$ on these. We say that the coupling is successful if $Y$ and $Z$ reach level $L$ without visiting any parts of $\Gn$ and $T$ that disagree or got truncated.

Let $K=C_K$ as defined in \eqref{eq:LK}, where $C_K$ is a sufficiently large constant, to be specified later.

For a long-range edge $e$ of a given quasi-tree $T$ and a given constant $A$ we define truncation criteria as follows.
\[
\trunc{e}{A}:=\quad\left\{-\log\prcond{\xi_{\ell(e)}=e}{\xi,T}{}>\frac12\log N+A\sqrt{(\log\log N)(\log N)}\right\},
\]
\[
\truncprime{e}:=\quad\left\{-\log\pr{\xi_{L-K}=e}<\frac12\log N+\sqrt{(\log\log N)(\log N)}\right\}.
\]

Then we can bound the probability of $Z$ crossing a truncated edge before hitting level $L-K$ as follows.

\begin{lemma}
For any $\theta>0$, for sufficiently large values of $A$ (in terms of $\theta$ and $C_L$), and sufficiently large $n$ we have
\[\pr{\bigcup_{k=1}^{L-K}\trunc{\xi_k}{A}}<\theta.\]
Also, for sufficiently large values of $C_{L,2}$ (in terms of $\theta$ and $C_L$), and sufficiently large $n$ we have
\[\pr{\truncprime{\xi_{L-K}}}<\theta.\]
\end{lemma}

%\begin{proof}[Sketch proof]
%We use that $\trunc{e}{A}$ implies $\trunc{e'}{A}$ for all descendants $e'$ of $e$, and~\Cref{entropy_concentration} with $k=L-K$.
%\end{proof}
%\magenta{delete above and instead write:
\begin{proof}
The proof follows using that $\trunc{e}{A}$ implies $\trunc{e'}{A}$ for all descendants $e'$ of $e$, and~\Cref{entropy_concentration} with $k=L-K$.
\end{proof}%}

\begin{definition}\label{def:K-root} Let $u$ be a vertex of $\Gn$ and let $e_u$ be the type $0.9n$ red edge from it. Then we say that $u$ is a $K$-root in $\Gn$ if a neighbourhood of $u$ of green and red edges in $\Gn\setminus\{e_u\}$ is a possible realisation of ($\iota$ of) the first $K$ levels of a quasi-tree. We denote the set of $K$-roots by $\VKroot$.
\end{definition}

Let us consider any two vertices $x$ and $y$ that are in different type $(n-1)$ subgraphs of $\Gn$ and work on the event that both of them are $K$-roots, with corresponding neighbourhoods $T_{x,0}$ and $T_{y,0}$ respectively. Then we couple the neighbourhoods of $x$ and $y$ with the first $L$ levels of two independent quasi-trees $T_x$ and $T_y$ conditioned to have the first $K$ levels as $T_{x,0}$ and $T_{y,0}$ respectively, as follows.

We denote the level $K$ vertices of $T_x$ by $z_1$, ..., $z_{L_x}$, and the level $K$ vertices of $T_y$ by $z_{L_x+1}$, ..., $z_{L_x+L_y}$. Then we explore from $z_1$ up to level $L$ as follows.

Firstly, we choose $k\in\{0.9n,...,n-1\}$ such that the type $k$ long-range edge of $v$ is not yet explored, we reveal the other endpoint $v$ of the type $k$ long-range edge from $z_1$ in $T_x$, and couple $\iota(v)$ with the other endpoint of the type $k$ red edge from $\iota(z_1)$ in $\Gn$ according to the optimal coupling. We also couple the $R$-ball of $v$ with the distance $R$ green neighbourhood of $\iota(v)$ in $\Gn$ according to the optimal coupling. If either of these optimal couplings fails, we say that the long-range edge $(z_1,v)$ is truncated due to the optimal coupling failing. Otherwise, if the distance $R$ green neighbourhood of $\iota(v)$ contains any of the already explored vertices, we say that $(z_1,v)$ is truncated due to an overlap. Otherwise, if $\trunc{(z_1,v)}{A}$ or $\truncprime{(z_1,v)}$ holds with respect to $T_x(z_1)$, we say that $(z_1,v)$ is truncated according to a truncation criterion. We proceed similarly with the other yet unexplored long-range edges from $z_1$. If a long-range edge $(z_1,v)$ is truncated for any of the above reasons, then we will not explore further into this direction. Otherwise we continue exploring analogously from each vertex in the $R$-ball of $v$. We continue the exploration up to level $L$. (In each level, the truncation criteria are considered with respect to $T_x(z_1)$.)

Then we explore similarly from $z_2$, $z_3$, ..., $z_{L_x+L_y}$. (When exploring from $z_i$ we consider the truncation criteria with respect to $T_x(z_i)$ or $T_y(z_i)$.)

We denote by $\F_i$ the $\sigma$-algebra generated by the explorations from $z_1$, ..., $z_i$. We say that $z_i$ is \emph{good} if it is not contained in any $R$-ball revealed during the explorations from $z_1$, ..., $z_{i-1}$. Otherwise we call $z_i$ \emph{bad}.

After coupling the neighbourhoods of $x$ and $y$ with the trees, we couple independent walks $Y^{(1)}$ and $Y^{(2)}$ on the green and red edges of $\Gn$ from $x$ and $y$ respectively, with independent walks $Z^{(1)}$ and $Z^{(2)}$ on $T_x$ and $T_y$ from $x$ and $y$ respectively such that the following are satisfied. Let $\taured^{(L,i)}$ denote the $L$th time that $Y^{(i)}$ crosses a red edge. We say that a vertex $v$ of $T$ is at the boundary of an $R$-ball if it is at graph distance $R$ from the centre of the $R$-ball it is contained in.

We couple $Y^{(1)}$ and $Y^{(2)}$ with $Z^{(1)}$ and $Z^{(2)}$ up to $\taured^{(K,1)}$ and $\taured^{(K,2)}$ respectively such that on the event that for $i=1,2$
\begin{enumerate}[(i)]
\item up to $\taured^{(K,i)}$ the walk $Z^{(i)}$ does not hit the boundary of an $R$-ball,
\item the first red edge that $Y^{(i)}$ crosses is not the type $0.9n$ %\magenta{I think it should be $0.9n$}
red edge from $Y^{(i)}_0$, and
\item up to the $K$th time that $Y^{(i)}$ crosses a red edge, it never backtracks the most recently crossed red edge,
\end{enumerate}
we have $Y^{(i)}_j=\iota\left(Z^{(i)}_j\right)$ for $i=1,2$, $j\le\taured^{(K,i)}$.

We say that the coupling is successful in the first $K$ levels if (i), (ii) and (iii) hold. In this case if $Z^{(i)}$ reached level $K$ at $z_{k_i}$, then we couple the rest of the walks $Y^{(1)}$ and $Y^{(2)}$ with $Z^{(1)}$ and $Z^{(2)}$ from $z_{k_1}$ and $z_{k_2}$ until $\taured^{(L,1)}$ and $\taured^{(L,2)}$ respectively such that on the event that for $i=1,2$
\begin{enumerate}[(i)]
\item[(iv)] between $\taured^{(K,i)}$ and $\taured^{(L,i)}$ the walk $Z^{(i)}$ does not cross a truncated long-range edge and does not hit the boundary of an $R$-ball,
\item[(v)] between the $K$th and the $L$th time that $Y^{(i)}$ crosses a red edge, it never backtracks the most recently crossed red edge,
\end{enumerate}
we have $Y^{(i)}_j=\iota\left(Z^{(i)}_j\right)$ for $i=1,2$, $\taured^{(K,i)}\le j\le\taured^{(L,i)}$.

We say that the coupling is successful between levels $K$ and $L$ if (iv) and (v) hold.

We say that the coupling is successful if it is successful in the first $K$ levels, and also between levels $K$ and $L$.

%\magenta{Add some figures here?}

\section{Estimates regarding $K$-roots and the success of coupling}\label{sec:success}

In this section we prove an estimate about $\VKroot$ that will make it sufficient to consider the case when $u$ and $v$ in~\eqref{eq:sum} are both $K$-roots, and then we estimate the probability that the coupling around two $K$-roots succeeds.

\subsection{Estimates regarding $K$-roots} \label{sec:K-root_estimates}

In this section we show that for any starting vertex, $Y_{\taured^K}$ is a $K$-root with high probability.

%\begin{lemma}\label{many_Kroots}
%With high probability $\Gn$ is such that $|\VKroot|=(1-o(1))N$.
%\end{lemma}

\begin{lemma}\label{hit_Kroot} For sufficiently large values of $C_K$ the following holds. With high probability $\Gn$ is such that for any vertex $q$ we have $\prcond{Y_{\taured^K}\in \VKroot}{\Gn}{q}=1-o(1)$.
\end{lemma}

\begin{proof}%[Proof of Lemmas~\ref{many_Kroots} and~\ref{hit_Kroot}]
Note that the number of vertices in a given $R$-ball is $\le2n^R$ and for any $k$ the number of vertices in the first $k$ levels of a quasi-tree is $\le n^{2kR}$. Also, if there are $s$ red edges revealed in $\Gn$ so far and we are revealing the red edge of a given type, from a given vertex, the probability of the other endpoint taking a specific value is $\le\frac{1}{2^{0.9n-1}-s}$. If $s\le2^{0.1n}$, then this is $\le4N^{-\frac{9}{10}}$.

Let us consider a given vertex $q$ of $\Gn$ and reveal its `level $2K$ neighbourhood' in $\Gn$ as follows. Firstly, let us reveal the green edges around it up to green distance $R$. Then, for each vertex in this ball, let us reveal the other endpoint of the red edges from it, except for the type $0.9n$ red edge from $q$. For each newly revealed endpoint, reveal the green edges around it up to green distance $R$. Then for each vertex in these newly revealed balls, let us reveal the other endpoint of the red edges emanating from it. Let us continue similarly for $2K$ levels. We say that an overlap occurs if a newly revealed ball of green distance $R$ intersects with any previously revealed ball.

%\magenta{Double check the following calculations.}

 Firstly, note that the probability of the walk backtracking any of the first {$K+1$} long-range edges crossed is $o(1)$, so whp crossing $K$ red edges corresponds to reaching long-range distance $K$. Also, from the above estimates we get that the probability that the level $2K$ neighbourhood of a given vertex $q$ contains more than one an overlap is $\le\pr{\Bin{n^{4RK}}{\frac{4n^{3RK}}{N^{\frac{9}{10}}}}>1}\lesssim\left(\frac{n^{5RK}}{N^{\frac{9}{10}}}\right)^2\ll\frac1N$. If there is only one overlap, then whp the first vertex reached by the walk that is at long-range distance $K$ from $q$ is a $K$-root. This is because the probability of reaching level $K$ at a vertex which is a descendant of a specific long-range edge $e$ leading to the first level is at most $o(1)$, as each vertex has $\Theta(n)$ long-range connections and the walk equally likely crosses any of them. If there is one overlap, then at most two long-range edges leading to the first level will have the overlap among one of its descendants (or with the level $0$ ball), meaning that as long as the walk avoids crossing to level one through one of these two edges it will be in a K-root when it reaches level $K$.

So whp $\Gn$ is such that the level $2K$ neighbourhood of each $q$ has at most 1 overlap, and on this event we have $\prcond{Y_{\taured^K}\in\VKroot}{\Gn}{q}=1-o(1)$.
\end{proof}

\subsection{Estimates regarding the success of coupling} \label{sec:coupling_estimates}

In this section we prove that with high probability $\Gn$ is such that for any two $K$-roots in different type $(n-1)$ subgraphs, the coupling described in~\Cref{sec:trunc_Kroot_coupling} has a positive probability of succeeding.

Note that in a given quasi-tree $T$ the number of non-truncated long-range edges at a given level is\\ $\le\exp\left(\frac12\log N+A\sqrt{(\log\log N)(\log N)}\right)=N^{\frac12+A\sqrt{\frac{\log\log N}{\log N}}}$. In $T_x$ and $T_y$ the number of vertices up to level $K$ is $\le n^{2KR}$ and the number of explored vertices in each level of each $T(z_i)$ is $\le N^{\frac12+A\sqrt{\frac{\log\log N}{\log N}}}$. Hence the overall number of explored vertices up to level $L$ is $\le Ln^{2KR}N^{\frac12+A\sqrt{\frac{\log\log N}{\log N}}}\ll N^{\frac12+\alpha}$ for any $\alpha>0$.

\begin{lemma}\label{num_bad}
There exists a constant $B$ with the following property.
Let $x$ and $y$ be two vertices in different type $(n-1)$ subgraphs of $\Gn$ and let $T_{x,0}$ and $T_{y,0}$ be possible realisations of the first $K$ levels of quasi-trees. Let $\{T_{x,0},T_{y,0}\}$ denote the event that some neighbourhoods of $x$ and $y$ in $\Gn$ look like $T_{x,0}$ and $T_{y,0}$ respectively. Then the number $\mathrm{Bad}$ of bad vertices (defined in~\Cref{sec:trunc_Kroot_coupling}) satisfies
\[
\prcond{\mathrm{Bad}>B}{T_{x,0},T_{y,0}}{}\quad\ll\quad\frac{1}{N^2}\:.\qedhere
\]
\end{lemma}

\begin{proof}
At each step of the exploration the probability of the newly revealed $R$-ball containing a $z_i$ is $\lesssim\frac{n^{2RK}n^{2R}}{N^{\frac{9}{10}}}$. This is $\ll\frac{1}{N^{\frac{9}{10}-\alpha}}$ for any $\alpha>0$. So
\[
\prcond{\textrm{Bad}>B}{T_{x,0},T_{y,0}}{}\quad\le\quad\binom{N^{\frac12+\alpha}}{B}\left(\frac{1}{N^{\frac{9}{10}-\alpha}}\right)^B\qquad\le\quad N^{-B\left(\frac25-2\alpha\right)},
\]
which is $\ll\frac{1}{N^2}$ for $\alpha$ sufficiently small and $B$ sufficiently large.
\end{proof}

\begin{lemma}\label{coupling_from_good_succeeds}
For any $\theta\in(0,1)$ for sufficiently large value of $C_R$ and sufficiently large values of $n$, for all $i\in\{1,2,...,L_x\}$, the following holds. Let us consider the coupling of $Y^{(1)}$ and $Z^{(1)}$ from $z_i$, between levels $K$ and $L$, as in \Cref{sec:trunc_Kroot_coupling}. Then we have
\[
\prcond{\text{coupling of }Y^{(1)}\text{ and }Z^{(1)}\text{ succeeds between levels }K\text{ and }L}{\F_{i-1}}{z_i}\quad\ge\quad(1-\theta)\1{\{z_i\text{ good}\}}.
\]
The analogous result holds for $i\in\{L_x+1,...,L_x+L_y\}$ with $Y^{(2)}$ and $Z^{(2)}$ instead of $Y^{(1)}$ and $Z^{(1)}$.
\end{lemma}

\begin{proof}
In case $z_i$ is good, we know that the $R$-ball of $z_i$ does not intersect any of the already revealed $R$-balls. Now consider the exploration from $z_i$ and the walks from $z_i$.

The probability that the walk $Z^{(j)}$ up to time $\taured^{(L,j)}$ crosses a long-range edge that is truncated due to an overlap is $\lesssim L n^{2R}\frac{L n^{2R} N^{\frac12+\alpha}}{N^{\frac{9}{10}}}\ll1$. Similarly, the probability of $Z^{(j)}$ crossing a long-range edge where the optimal coupling failed is also $\ll1$.

The probability of $Z^{(j)}$ hitting the boundary of an $R$-ball is $\lesssim L\pr{\Geomnonneg{\frac{0.1n-1}{n}}>R}\lesssim\frac{n}{\log n}\left(\frac{9.1}{10}\right)^{C_5\log n}$. This is $\ll1$ for sufficiently large values of $C_R$.

The probability that up to the $L$th time that $Y^{(j)}$ crosses a red edge, it ever backtracks the most recently crossed red edge is $\le L\frac{1}{0.1n}\ll1$.
\end{proof}

\begin{lemma}\label{coupling_up_to_K_succeeds}
For sufficiently large values of $C_R$ the following holds. For any realisation of the first $K$ levels of $T_{x,0}$ and $T_{y,0}$, the probability that the coupling of $Z^{(i)}$ and $Y^{(i)}$ is successful in the first $K$ levels is $1-o(1)$ .
\end{lemma}

\begin{proof}
The probability that the first red edge crossed by $Y^{(j)}$ is the type $0.9n$ one from $Y^{(j)}_0$ is $\le\frac{1}{0.1n}\ll1$.

The probabilities of the other ways of the coupling failing can be bounded analogously to the proof of \Cref{coupling_from_good_succeeds}.
\end{proof}

\begin{lemma}\label{coupling_succeeds}
For any $\theta\in(0,1)$, for sufficiently large values of the constants $C_R$ and $C_K$, the following holds. For $x$, $y$ and $\{T_{x,0},T_{y,0}\}$ as in~\Cref{num_bad} we have
\[
\prcond{\prcond{\text{the coupling of }Y^{(1)}\text{ and }Z^{(1)}\text{ succeeds}}{\F_{L_x+L_y}}{x}>1-\theta}{T_{x,0},T_{y,0}}{}\quad\ge\quad1-o\left(\frac{1}{N^2}\right).
\]
The analogous result holds for $Y^{(2)}$ and $Z^{(2)}$ starting from $y$.
\end{lemma}

\begin{proof}
Let
\[
\widehat{V}:=\quad\left\{z_i:\:\prcond{\xi_K^+=z_i,\text{ coupling of }X^{(1)}\text{ and }Z^{(1)}\text{ fails}}{\F_{L_x+L_y}}{x}\ge\theta\right\},
\]
\[h(z_i):=\quad\prcond{\xi_K^+=z_i}{\F_{L_x+L_y}}{x},\]
where $\xi_K^+$ denotes the endpoint of $\xi_K$ further from the root.

We will show that
\begin{align}\label{eq:bound_Vhat}
\prcond{h(\widehat{V})>3\theta}{T_{x,0},T_{y,0}}{}\quad\ll\quad\frac{1}{N^2}.
\end{align}

Once we have this, the proof is immediate (we have to replace $\theta$ by $\frac15\theta$).

The proof of~\eqref{eq:bound_Vhat} is analogous to the proof of~\cite[Proposition 5.7]{random_matching}.
We use that $h(z_i)\le\left(\frac{1}{0.1n-1}\right)^K\ll\frac{1}{n^2}$.
\end{proof}

\section{Bounding the hit time}\label{sec:hit}

In this section we conclude the proof of \Cref{transition_probs}, and then the proof of \Cref{hit_upper_bound}.

\subsection{Proof of \Cref{transition_probs}}

For $u,v\in\VKroot$ in different type $(n-1)$ subgraphs of $\Gn$, let $\F_{u,v}=\F_{L_u+L_v}$ be the $\sigma$-algebra generated by the explorations around them up to level $L$ and let\\ $\Omega_{u,v}=\{\text{coupling of }(Y^{(1)},Y^{(2)})\text{ and }(Z^{(1)},Z^{(2)})\text{ from }(u,v)\text{ succeeds}\}$.

\begin{lemma}\label{prob_between_Kroots}
For any $\theta\in(0,1)$, there exists $\theta'\in(0,1)$ with the following property. For any $u$ and $v$ in different type $(n-1)$ subgraphs of $\Gn$, on the event that $u,v\in\VKroot$ and $\prcond{\Omega_{u,v}}{\F_{u,v}}{}>\theta$, we have
\begin{align*}
\prcond{\sum_{w,z}\prcond{Y^{(1)}_{\taured^{(L,1)}-1}=w,Y^{(2)}_{\taured^{(L,2)}-1}=z,\Omega_{u,v}}{\Gn}{(u,v)}\1{\eta(w)=z}<\frac{\theta'}{N}}{\F_{u,v}}{}\quad\ll\quad\frac{1}{N^2}\:.
\end{align*}
\end{lemma}

\begin{proof}
By the definition of $\truncprime{\cdot}$ we have \[\prcond{Y^{(1)}_{\taured^{(L,1)}-1}=w,\Omega_{u,v}}{\F_{u,v}}{v}\le\exp\left(-\frac12\log N-\sqrt{(\log N)(\log\log N)}\right)<\frac{1}{\sqrt{N}\log N}\] for all $w$, and similarly $\prcond{Y^{(2)}_{\taured^{(L,2)}-1}=z,\Omega_{u,v}}{\F_{u,v}}{v}<\frac{1}{\sqrt{N}\log N}$ for all $z$. Also, conditional on $\F_{u,v}$, the type $n$ blue edges form a uniform random matching between the two type $(n-1)$ subgraphs.

For each $i$ and $j$ in these type $(n-1)$ subgraphs let $w_{i,j}=\prcond{Y^{(1)}_{\taured^{(L,1)}-1}=i,Y^{(2)}_{\taured^{(L,2)}-1}=j,\Omega_{u,v}}{\F_{u,v}}{(u,v)}$. Also, let $m=2a=\frac{1}{2^{n-1}}\sum_{i,j}w_{i,j}\asymp\frac{1}{N}$ and $b=\max_{i,j}w_{i,j}\le\frac{1}{N(\log N)^2}$. Then with a proof analogous to the proof of~\cite[Lemma 5.1]{cutoff_NBRW_on_sparse_random_graphs} we get that
\begin{align*}
\prcond{\sum_{i}w_{i,\eta_n(i)}\le m-a}{\F_{u,v}}{}\quad\le\quad\exp\left(-\frac{a^2}{4bm}\right)\quad\ll\quad\frac{1}{N^2}.
\end{align*}
Note that the probability $\prcond{Y^{(1)}_{\taured^{(L,1)}-1}=w,Y^{(2)}_{\taured^{(L,2)}-1}=z,\Omega_{u,v}}{\Gn}{(u,v)}$ is $\F_{u,v}$-measurable, so we can change the conditioning on $\Gn$ to a conditioning on $\F_{u,v}$. This finishes the proof.
\end{proof}

Let us choose $C_{L,2}$, $C_R$ and $C_K$ such that Lemmas~\ref{hit_Kroot} and~\ref{coupling_succeeds} are satisfied.

Let $\theta_1$ be the constant from \Cref{hit_Kroot}, let $\theta_2$ be an arbitrary constant in $(0,1)$ and let $\theta_3$ be the corresponding constant $\theta'$ from \Cref{prob_between_Kroots}.

Let $\Omega_0$ be the high probability event that $\Gn$ satisfies all of the following.
\begin{itemize}
\item For all $q$ we have $\prcond{Y_{\taured^K}\in \VKroot}{\Gn}{q}\ge\theta_1$. (See \Cref{hit_Kroot}.)
%\item $|\VKroot|\ge\frac12N$ for all sufficiently large $n$. (See~\Cref{many_Kroots}.)
\item For all $u,v\in\VKroot$ in different type $(n-1)$ subgraphs, we have $\prcond{\Omega_{u,v}}{\F_{u,v}}{}>\theta_2$. (See~\Cref{coupling_succeeds}.)
\item  For all $u,v\in\VKroot$ in different type $(n-1)$ subgraphs with $\prcond{\Omega_{u,v}}{\F_{u,v}}{}>\theta_2$, we have $\sum_{w,z}\prcond{Y^{(1)}_{\taured^{(L,1)}-1}=w,Y^{(2)}_{\taured^{(L,2)}-1}=z,\Omega_{u,v}}{\Gn}{(u,v)}\1{\eta(w)=z}\ge\frac{\theta_3}{N}$. (See~\Cref{prob_between_Kroots}.)
\end{itemize}

\begin{proof}[Proof of \Cref{transition_probs}]
Let us work on the high probability event $\Omega_0$.

Consider any $x$ and $y$ that are in different type $(n-1)$ subgraphs. Let $\ell$ be such that $x$ is in $G^{(n-1,\ell)}$. Then by~\eqref{eq:sum} we have
\begin{align*}
&\prcond{\til{X}_{\tau}=y}{\Gn}{x}\\
&\ge\quad\left(\inf_{q\in G^{(n-1,\ell)}} \prcond{Y_{\taured^K}\in\VKroot}{\Gn}{q}\right) \prcond{Y_{\taured^K}\in\VKroot}{\Gn}{y}\\
&\hspace{2cm}\cdot\left(\inf_{\substack{u\in\VKroot\cap G^{(n-1,\ell)}\\ v\in\VKroot\cap G^{(n-1,3-\ell)}}} \sum_{w,z}\prcond{Y_{\taured^L-1}=w}{\Gn}{u}\prcond{Y_{\taured^L-1}=z}{\Gn}{v}\1{\eta(w)=z}\right).
\end{align*}
By the definition of $\Omega_0$ this is $\ge\theta_1^2\theta_3\frac1N$. This finishes the proof.
\end{proof}

\subsection{Proof of \Cref{hit_upper_bound}}

%\begin{proof}[Proof of~\Cref{hit_upper_bound}]
Let $\theta$ be as in \Cref{transition_probs} and let us work on the corresponding high probability event. Then we have
\begin{align*}
1-\dtv{\prcond{\til{X}_{\tau}=\cdot}{\Gn}{x}}{\U(\cdot)}\quad=\quad\sum_{y}\frac1N\wedge\prcond{\til{X}_{\tau}=y}{\Gn}{x}\quad\ge\quad N\cdot\frac{\theta}{N}=\quad\theta,
\end{align*}
hence, also using \Cref{couple_X_Xtil} (with $\frac12\theta$), we get
\begin{align*}
\dtv{\prcond{X_{\tau}=\cdot}{\Gn}{x}}{\U(\cdot)}\quad\le\quad\prcond{X_{\tau}\ne\til{X}_{\tau}}{\Gn}{x}\quad+&\quad\dtv{\prcond{\til{X}_{\tau}=\cdot}{\Gn}{x}}{\U(\cdot)}\\
&\qquad\le\quad\frac12\theta+\left(1-\theta\right)\quad=\quad1-\frac12\theta.
\end{align*}

From \Cref{tau_upper_bound} we know that there exists $C$ such that $\prcond{\tau>Cn}{\Gn}{x}\le\frac{1}{4}\theta$ for all $x$.

Let $\alpha=1-\frac{1}{8}\theta$ and let us consider any $x\in\Vn$ and $A\subseteq\Vn$ with $\U(A)\ge\alpha$.

Then $\prcond{\tau_A>Cn}{\Gn}{x}\le\prcond{\tau>Cn}{\Gn}{x}+\prcond{\tau_A>\tau}{\Gn}{x}$. We have $\prcond{\tau>Cn}{\Gn}{x}\le\frac{1}{4}\theta$  and $\prcond{\tau_A>\tau}{\Gn}{x}\le \prcond{X_\tau\not\in A}{\Gn}{x}\le \U(A^c)+\dtv{\prcond{X_{\tau}=\cdot}{\Gn}{x}}{\U(\cdot)}\le(1-\alpha)+\left(1-\frac18\theta\right)$. So overall we get that $\prcond{\tau_A>Cn}{\Gn}{x}\le1-\frac{1}{8}\theta$

This shows that $\hit{1-\frac{1}{8}\theta}{1-\frac{1}{8}\theta}\le Cn$, and so finishes the proof.
%\end{proof}

\section{Bounding the absolute relaxation time}\label{sec:trelabs}

In this section we prove \Cref{trelabs_order}.

We will prove that whp the graph $\Gn$ is such that for any partition $(A,B)$ of the vertices, at least a constant proportion of the edges of $\Gn$ run within $A$ or within $B$. Then we show that this property implies $\trelabs\asymp\trel$.

First we prove the following lemma.

\begin{lemma}\label{three_levels}
There exists a positive constant $c$ with the following property.

Let $m\in\Zpos$ and let $V_1$, $V_2$, ... $V_8$ be disjoint sets of size $m$ each. Let us consider a random graph $H$ with vertex set $V=\bigcup_{i=1}^{8}V_i$ and edge set $E$ consisting of a uniformly chosen perfect matching between $V_{2i-1}$ and $V_{2i}$ for $i=1,2,3,4$, a uniform perfect matching between $V_{4i-3}\cup V_{4i-2}$ and $V_{4i-1}\cup V_{4i}$ for $i=1,2$ and a uniform perfect matching between $\bigcup_{i=1}^{4}V_i$ and $\bigcup_{i=5}^{8}V_i$. (These matchings are independent of each other.)

Then with probability $\ge 1-2^{-\frac14m}$ the graph $H$ is such that for any partition $(A,B)$ of $V$ we have $|E(A)|+|E(B)|\ge c|E|$. Here $E(A)$ denotes the set of edges with both endpoints in $A$, and $E(B)$ is defined analogously.
\end{lemma}

\begin{proof}
Consider a partition $(A,B)$ of $V$. Let $A_0:=A\cap(V_1\cup V_2)$, $B_0:=B\cap(V_1\cup V_2)$. Let $\til{A}$ consist of $A_0$ and the neighbours of $B_0$ in $V_3\cup V_4$. Let $\til{B}=\left(\bigcup_{i=1}^{4}V_i\right)\setminus\til{A}$. Let $\ha{A}$ consist of the neighbours of $\til{B}$ in $\bigcup_{i=5}^{8}V_i$, and let $\ha{B}=\left(\bigcup_{i=5}^{8}V_i\right)\setminus\ha{A}$.
Fix some constant $\delta\in\left(0,\frac14\right)$ such that $C(\delta)$ in~\cite[Lemma 6.5]{random_matching} satisfies $C(\delta)\le\frac18$.

We will consider the following events.
\begin{align*}
\Omega_1:=&\quad\left\{|A_0|\not\in\left(\frac12-\delta,\frac12+\delta\right)|V_1\cup V_2|\right\},\\
\Omega_2:=&\quad\left\{\ge2\delta m\text{ edges run between }\til{B}\cap V_3\text{ and }\til{B}\cap V_4\right\},\\
\Omega_3:=&\quad\left\{\ge4\delta m\text{ edges run between }\ha{B}\cap(V_5\cup V_6)\text{ and }\ha{B}\cap(V_7\cup V_8)\right\}.
\end{align*}
Note that these events only depend on $A_0$ and the graph $H$. We will show that if any of them holds, then for any choice of $(A,B)$ with the given $A_0$, there are $\ge\delta m$ edges with both endpoints in $A$ or both endpoints in $B$. Then we will use~\cite[Lemma 6.5]{random_matching} to bound the probability of these events failing for a given $A_0$, and take a union bound over the possible choices of $A_0$.

If $\Omega_1$ holds, then by pigeonhole principle at least $\delta m$ edges of the matching between $V_1$ and $V_2$ have both endpoints in $A_0$ or both endpoints in $B_0$.

If $\Omega_2$ holds and $|\til{B}\cap A\cap(V_3\cup V_4)|\le\delta m$ then at least $\delta m$ of the edges between $\til{B}\cap V_3$ and $\til{B}\cap V_4$ have both endpoints in $B$. If $\Omega_2$ holds and $|\til{B}\cap A\cap(V_3\cup V_4)|>\delta m$ then there are at least $\delta m$ edges between $A_0$ and $\til{B}\cap A\cap(V_3\cup V_4)$. So in either case there are $\ge\delta m$ edges with both endpoints in $A$ or both endpoints in $B$.

If $\Omega_3$ holds and $|\ha{B}\cap A|\le2\delta m$ then at least $2\delta m$ of the edges between $\ha{B}\cap(V_5\cup V_6)$ and $\ha{B}\cap(V_7\cup V_8)$ have both endpoints in $B$. If $\Omega_3$ holds and $|\ha{B}\cap A|>2\delta m$ then there are at least $2\delta m$ edges between $\til{A}$ and $\ha{B}\cap A$. If $|\til{A}\cap B|\le \delta m$ then at least $\delta m$ of these edges have both endpoints in $A$. If $|\til{A}\cap B|>\delta m$ then there are at least $\delta m$ edges between $B_0$ and $\til{A}\cap B$. So in all cases there are $\ge\delta m$ edges with both endpoints in $A$ or both endpoints in $B$.

In what follows assume that $|A_0|\in\left(\frac12-\delta,\frac12+\delta\right)|V_1\cup V_2|$. Also note that $|\til{A}|=\frac12\left|\bigcup_{i=1}^{4}V_i\right|$.

Let $\sigma$ be a matching on $V_1\cup V_2$ that matches vertices $x$ and $y$ if and only if their neighbours in $V_3\cup V_4$ are connected by an edge. Note that $\sigma$ is a uniformly chosen perfect matching on $V_1\cup V_2$. (The uniform randomness of $\sigma$ comes from the randomness of edges between $V_1\cup V_2$ and $V_3\cup V_4$, for any realisation of the other edges.) By~\cite[Lemma 6.5]{random_matching} we know that with probability $\ge1-2^{-2m\left(\frac12-C(\delta)\right)}$ there are at least $2\delta m$ pairs in $\sigma$ with both ends in $A_0$, i.e. $\Omega_2$ holds.

Now consider a matching $\eta$ of $\bigcup_{i=1}^{4}V_i$ that matches vertices $x$ and $y$ if and only if their neighbours in $\bigcup_{i=5}^{8}V_i$ are connected by an edge between $V_{5}\cup V_{6}$ and $V_{7}\cup V_{8}$. This is a uniform perfect matching on $\bigcup_{i=1}^{4}V_i$, independently of $\sigma$. By~\cite[Lemma 6.5]{random_matching} we know that with probability $\ge1-2^{-4m\left(\frac12-C(\delta)\right)}$ there are at least $4\delta m$ pairs in $\eta$ with both ends in $\til{A}$, i.e.\ $\Omega_3$ holds.

Taking a union bound over the $\le2^{2m}$ choices of $A_0\subseteq(V_1\cup V_2)$ with $|A_0|\in\left(\frac12-\delta,\frac12+\delta\right)|V_1\cup V_2|$ finishes the proof.
%So for any $A_0\subseteq(V_1\cup V_2)$ with $|A_0|\in\left(\frac12-\delta,\frac12+\delta\right)|V_1\cup V_2|$, the probability of both $\Omega_1$ and $\Omega_2$ failing is $\le2^{-6m\left(\frac12-C(\delta)\right)}$. There are $\le 2^{2m}$ choices of such $A_0$, so taking a union bound over them finishes the proof.
\end{proof}

We use the above lemma for different subgraphs of $\Gn$ to show the following.

\begin{lemma}\label{lemma:edge_inner_boundary} There exists a positive constant $c'$ with the following property.
With high probability $\Gn$ is such that for any partition $(A,B)$ of its vertices we have $|E(A)|+|E(B)|\ge c'|E|$.
\end{lemma}

\begin{proof}
Let $\frac{n}{2}+5\le m\le n$ and consider the set of edges $E_{m,i}$ at levels $m-2$, $m-1$ and $m$ in graph $G^{(m,i)}$ for each $i\in\{1,2,...,2^{n-m}\}$. Using \Cref{three_levels} for each $E_{m,i}$ and taking a union bound we get that with probability $\ge1-2^{n-m}\cdot2^{-\frac14\cdot2^{m-3}}\ge1-2^{\frac{n}{2}-2^{\frac{n}{2}}}$ each $E_{m,i}$ is such that for any partition $(A,B)$ of the vertices at least $c$ proportion of the edges have both endpoints in $A$ or both endpoints in $B$.

Taking a union bound over $m$ and using that a constant proportion of the edges of $\Gn$ is at levels $\ge\frac{n}{2}+3$ we get that with probability $\ge 1-\frac{n}{2}\cdot2^{\frac{n}{2}-2^{\frac{n}{2}}}=1-o(1)$ the graph $\Gn$ has the required property.
\end{proof}

Finally, we use a result stating that the above property implies that $\trel\asymp\trelabs$.

\begin{proposition}\label{trelabsgeneral} Let $G=(V,E)$ be a graph satisfying that there is a constant $c>0$ such that for any partition of $V$ into two sets $A$ and $B$, we have $|E(A)|+|E(B)|>c|E|$. %, where for $S\subset V$, $E(S)$ is the set of edges with both endpoints in $S$.
Then a simple random walk on $G$ satisfies that $\trel\asymp\trelabs$.   
\end{proposition}

The proof of \Cref{trelabsgeneral} relies on a result from~\cite{near_bipartiteness} and it is presented in~\cite{LRP}.

%Now we are ready to conclude the proof of~\Cref{trelabs_order}

\begin{proof}[Proof of~\Cref{trelabs_order}]
The proof follows directly from Lemma~\ref{lemma:edge_inner_boundary}, Proposition~\ref{trelabsgeneral} and as from~\cite[Proposition 3]{twisted_hypercubes_structure_randomness} we know that $\trel\asymp n$.
\end{proof}

\section{The cover time}\label{sec:cover}

In this section we present the proof of \Cref{thm:main_cover}.

In what follows we consider a simple random walk. Once we prove the bounds for this, the bounds for a lazy random walk follow by noting that the cover time of a lazy walk is twice the cover time of a simple random walk.

The lower bound follows from a general lower bound from~\cite{tcov_lower_bound} which states that the cover time of a simple random walk on a graph with $N$ vertices satisfies $\tcov\ge N\log N$.

For the upper bound, we use Matthew's theorem (see~\cite[Theorem 11.2]{MTMC_book}) and a standard bound on the maximal hitting time in terms of the maximal hitting time from the invariant distribution (see~\cite[Lemma 10.2]{MTMC_book}, which give that $\tcov\le \max_{x,y\in \Vn}\estart{\tau_x}{y}\log N\le 2\max_{x\in \Vn}\estart{\tau_x}{\pi}\log N $, where $\tau_x$ denotes the hitting time of vertex $x$.

Therefore, it only remains to prove that
\begin{align}\label{eq:hit_from_pi_upper_bound}
\max_{x\in \Vn}\estart{\tau_x}{\pi}\quad\lesssim\quad N\quad=\quad2^n.
\end{align}

We know that $\estart{\tau^+_x}{x}=\frac{1}{\pi(x)}=2^n$ for all $x$, where $\tau_x^+$ denotes the first return time to $x$, and we will use this to prove an upper bound of the same order on any hitting time. 

In what follows we work on the high probability event that the mixing times of $\Gn$, $G^{(n-1,1)}$ and $G^{(n-1,2)}$ all satisfy the bounds in \Cref{thm:main_mixing}. Let us fix a vertex $x$ of $\Gn$ and for notational convenience assume that $x$ is in $G^{(n-1,1)}$.

Note that in each step a random walk has probability $\le\frac1n$ of stepping to $x$ and probability $\frac1n$ of stepping to $G^{(n-1,2)}$, hence for any positive constant $A$ there exists a positive constant $\eps_A$ such that $\prstart{\tau_x^+\ge An, \tau_{G^{(n-1,2)}}\le An}{x}\ge\eps_A$.

Let $Y$ be a simple random walk on $G^{(n-1,2)}$ and let $t:=2\tmixtext^{G^{(n-1,2)}}\left(\frac{1}{16}\right)$ be two times its $\frac{1}{16}$-mixing time. %\magenta{[Originally we had a $\frac18$-mixing time, check if that is sufficient.]}
From the assumption that $G^{(n-1,2)}$ satisfies the bounds in \Cref{thm:main_mixing} we know that $t\le Cn$ where $C$ is a positive constant. From~\cite[Lemma 6.17]{MTMC_book} we also know that at time $t$ the separation distance $s(t):=\max_{y,z\in V^{(n-1,2)}} \left(1-\frac{\prstart{Y_t=z}{y}}{\pi_{G^{(n-1,2)}}(z)}\right)$ of $Y$ satisfies $s(t)\le\frac14$. Hence we get that $\prstart{Y_t=z}{y}\ge \frac{3}{4}\pi_{G^{(n-1,2)}}(z)=\frac{3}{2N}$ for all $y$ and $z$ in $G^{(n-1,2)}$.

Let $X$ be a simple random walk on $\Gn$. Note that in each step $X$ crosses a type $n$ edge with probability $\frac1n$, regardless of its location, hence starting from $y$, we have $\tau_{G^{(n-1,1)}}\sim\Geompos{\frac1n}$ and conditional on the value of $\tau_{G^{(n-1,1)}}$, the walk $(X_k)_{k<\tau_{G^{(n-1,1)}}}$ is distributed like a simple random walk on $G^{(n-1,2)}$. Let $\theta$ be a sufficiently small constant such that $\pr{\Geompos{\frac1n}>t}\ge\theta$. Then we have
\begin{align*}
N\quad=\quad \estart{\tau_x^+}{x}\quad\ge&\quad \sum_{y,z\in V^{(n-1,2)}}\prstart{\tau_x^+\ge An, \tau_{G^{(n-1,2)}}\le An, X_{ \tau_{G^{(n-1,2)}}}=y}{x}\\
& \hspace{5cm} \cdot\prstart{X_t=z, t < \tau_{G^{(n-1,1)}}}{y}\cdot \estart{\tau_x}{z}\\
\ge&\quad \sum_{y,z\in V^{(n-1,2)}}\prstart{\tau_x^+\ge An, \tau_{G^{(n-1,2)}}\le An, X_{ \tau_{G^{(n-1,2)}}}=y}{x}\\
& \hspace{5cm} \cdot\prstart{Y_t=z}{y}\prstart{t < \tau_{G^{(n-1,\ell)}}}{y}\cdot \estart{\tau_x}{z}\\
\ge&\quad\sum_{y,z\in V^{(n-1,2)}}\prstart{\tau_x^+\ge An, \tau_{G^{(n-1,2)}}\le An, X_{ \tau_{G^{(n-1,2)}}}=y}{x}\\
& \hspace{5cm} \cdot  \frac{3}{2N}\pr{\Geompos{\frac1n}> t}\estart{\tau_x}{z}\\
\ge&\quad\sum_{z\in  V^{(n-1,2)}} \eps_A\frac{3}{2N}\theta\estart{\tau_x}{z}\quad= \quad\frac{3}{4}\eps_A\theta \estart{\tau_x}{\pi_{G^{(n-1,2)}}}.
\end{align*}

This shows that $\estart{\tau_x}{\pi_{G^{(n-1,2)}}}\le\frac{4}{3\eps_A\theta}N$.

To obtain a bound on $\estart{\tau_x}{\pi}$, note that $\prstart{X_{\tau_{G^{(n-1,2)}}}=y}{\pi_{G^{(n-1,1)}}}=\pi_{G^{(n-1,2)}}(y)$, hence

\begin{align*}
\estart{\tau_x}{\pi}\quad=\quad\frac12\estart{\tau_x}{\pi_{G^{(n-1,1)}}}+\frac12\estart{\tau_x}{\pi_{G^{(n-1,2)}}}\quad\le\quad& \estart{\tau_x}{\pi_{G^{(n-1,2)}}}+\frac12\estart{\tau_{G^{(n-1,1)}}}{\pi_{G^{(n-1,1)}}}\\
&\le\quad \frac{4}{3\eps_A\theta}N+\frac12n\quad\le\quad\left(\frac{4}{3\eps_A\theta}+1\right)N.
\end{align*}

This finishes the proof of \eqref{eq:hit_from_pi_upper_bound}, hence completing the proof of \Cref{thm:main_cover}.

\section{The chromatic number}\label{sec:chromatic}

In this section we prove the following auxiliary statement, which immediately implies \Cref{thm:main_chromatic}.

\begin{lemma}\label{aux_chromatic}
For any positive integer $c$ there exists a graph $H_c$ such that $\chi\left(H_c\right)>c$ and with high probability the graph $\Gn$ contains a copy of $H_c$.
\end{lemma}

We prove \Cref{aux_chromatic} by induction on $c$.

For $c=1$ the statement holds with $H_1$ consisting of two vertices connected via an edge.

Assume that the statement holds for $c=k$ and let us construct $H_{k+1}$ as follows.

Let $M_k$ be the number of vertices of $H_k$. Let $A$ be a set of $kM_k$ vertices and let $H_k^{(1)}$, ..., $H_k^{\left(\binom{kM_k}{M_k}\right)}$ be $\binom{kM_k}{M_k}$ disjoint copies of $H_k$. Let $A^{(1)}$, ..., $A^{\left(\binom{kM_k}{M_k}\right)}$ be an enumeration of all size $M_k$ subsets of $A$, and let $H_{k+1}$ be a graph on $kM_k+\binom{kM_k}{M_k}$ vertices obtained by considering $A$ and all $H^{(i)}$, and for each $i$ adding $M_k$ edges according to a perfect matching between the $M_k$ vertices in $A^{(i)}$ and the $M_k$ vertices in $H^{(i)}$.

Firstly, we show that $\chi\left(H_{k+1}\right)>k+1$. Assume for contradiction that there is a proper colouring of $H_{k+1}$ with $(k+1)$ colours. Then by the pigeonhole principle there is an $i$ such that all $M_k$ vertices in $A^{(i)}$ are of the same colour. \footnote{We use that $M_k\ge k+1$, hence $kM_k\ge(k+1)(M_k-1)+1$.} Since each vertex in $H_k^{(i)}$ is connected to a vertex in $A^{(i)}$, this colour cannot appear in the vertices of $H_k^{(i)}$. This means that we have a proper colouring of $H_k^{(i)}$ with $k$ colours, which contradicts $\chi\left(H_k\right)>k$. So we must have $\chi\left(H_{k+1}\right)>k+1$.

Now we turn to showing that whp $\Gn$ contains a copy of $H_{k+1}$.

We know that whp $\Gn$ contains a copy of $H_k$, hence there is a positive integer $n_k$ such that $\pr{G^{(n_k)}\text{ contains a copy of }H_k}\ge\frac12$. Let $m=n_k+\binom{kM_k}{M_k}$. We will show that\\ $\pr{G^{(m)}\text{ contains a copy of }H_{k+1}}\ge2^{-(mkM_k+1)\binom{kM_k}{M_k}}>0$. Once we have this, using that the number of independent copies of $G^{(m)}$ in $G^{(n)}$ diverges as $n\to\infty$, we can conclude that $\Gn$ contains $H_{k+1}$ with high probability.

We can prove the lower bound on $\pr{G^{(m)}\text{ contains a copy of }H_{k+1}}$ as follows.

Let $A$ be a fixed set of size $kM_k$ in $G^{(n_k,1)}$ and let $A^{(0)}$, ..., $A^{\left(\binom{kM_k}{M_k}-1\right)}$ be an enumeration of its size $M_k$ subsets. Note that for each $i$ we have $\pr{G^{(n_k+i,2)}\text{ contains a copy of }H_k}\ge\frac12$, since $G^{(n_k+i,2)}$ contains copies of $G^{(n_k)}$.

On the event that $G^{(n_k+i,2)}$ contains a copy $H_k^{(i)}$ of $H_k$ for each $i\in\left\{0,1...,\binom{kM_k}{M_k}\right\}$, the probability that the vertices of each $H_k^{(i)}$ are matched to the vertices of the corresponding $A^{(i)}$ is $\prod_{i=0}^{\binom{kM_k}{M_k}-1}\binom{2^{n_k+i}}{kM_k}^{-1}\ge2^{-mkM_k\binom{kM_k}{M_k}}$. We used here that for each $i$ there is a perfect matching between the vertices of $G^{(n_c+i,2)}$ and the vertices of $G^{(n_c+i,1)}$, the latter graph contains $G^{(n_k,1)}$ containing $A$, and these matchings are independent for different $i$.

This finishes the proof that $H_{k+1}$ has the desired properties, hence finishing the proof of \Cref{aux_chromatic}.

\begin{remark}
The above proof, if carried out with a bit more care, can provide a quantitative lower bound on the growth of the chromatic number as $n\to\infty$, however since this bound would be extremely weak, we decided not to work out its details. It remains an open question to find a sensible lower bound on $\chi\left(\Gn\right)$ as a function of $n$ that holds with high probability.
\end{remark}

\subsection*{Acknowledgements}

We would like to thank Itai Benjamini and Renan Gross for drawing our attention to this model. We would also like to thank Jonathan Hermon, Allan Sly and Perla Sousi for the very helpful discussions.

Zsuzsanna Baran was supported by DPMMS EPSRC DTP and by the ERC consolidator grant CUTOFF (101123174). Views and opinions expressed are however those of the authors only and do not necessarily reflect those of the European Union or the European Research Council Executive Agency. %\blue{[Who asked about the cover time?]}

\bibliography{bib_twisted_hypercube.bib}
\bibliographystyle{plain}

\end{document}